\documentclass[11pt]{article}
\usepackage{amssymb,amsmath,amsthm, wasysym}
\usepackage{amsfonts,amscd,graphicx, verbatim}
\usepackage{url}
\usepackage[margin=1in]{geometry}

\usepackage[utf8]{inputenc}
\author{Jakub Witaszek}
\title{The degeneration of the Grassmannian into a toric variety and the calculation of the eigenspaces of a torus action\footnote{The paper is based on the BSc thesis written at the Faculty of Mathematics, Informatics and Mechanics at Warsaw University under the supervision of prof. Jarosław Wiśniewski.}}

\date{September 2012}

\newtheorem{defi}{Definition}[subsection]
\newtheorem{theorem}[defi]{Theorem}
\newtheorem{cor}[defi]{Corollary }
\newtheorem{lem}[defi]{Lemma}
\newtheorem{prop}[defi]{Proposition}

\def\cc{\mathbb{C}}
\def\ccs{\cc^*}
\def\zz{\mathbb{Z}}
\def\nn{\mathbb{Z}_{\geq 0}}
\def\pp{\mathbb{P}}

\def\oo{\mathcal{O}}

\def\tt{\mathcal{T}}
\def\gg{\mathcal{G}}
\def\ntor{(\cc^*)^n}
\def\Xtt{X(\tt)}
\def\Xaff{X_{a\! f\! f}(\tt)}
\def\S{S(\tt)}
\def\Sn{S(\tt_n)}
\def\Snp{S(\tt_{n+1})}

\begin{document}

\maketitle
\begin{abstract}
	Using the method of degenerating a Grassmannian into a toric variety, we calculate recursive formulas for the dimensions of the eigenspaces of the action of an $n$-dimensional torus on a Grassmannian of planes in an $n$-dimensional space. In order to verify our result we compare it with the polynomial describing the Euler characteristic of invertible sheaves on a projective space with four blown-up points.
\end{abstract}



\section{Introduction}
	An $n$-dimensional torus acts on a Grassmannian of planes in an $n$ dimensional linear space in a natural way. In the paper I show a method of calculating the dimensions of the eigenspaces of the action of this torus on the coordinate ring of the Grassmannian. The result is presented as a generating function (Poincare-Hilbert series).\\

	Analysing the action of the torus may lead to a better understanding of Grassmannians - the varieties which are pervasive in mathematics and of significant importance. My method of calculating the dimensions of the eigenspaces of the action of the $n$-dimensional torus is based on the degeneration of the Grassmannian into a toric variety (see $\cite{Deg}$). Additionally, this method uses the properties of 3-valent trees and other combinatoric objects, which are interesting from the perspective of discrete mathematics.   

	The dimensions of the eigenspaces of the torus action are equal to the dimensions of the spaces of global sections of invertible sheaves on $n-3$ dimensional projective spaces with $n-1$ blown-up points (see $\cite{Muk}$). It follows from the fact that the coordinate ring of a Grassmannian may be interpreted as a ring of total coordinates of this variety. \\

	The paper is organised in the following way.

	In the second chapter, which consists of well known facts, I present the standard method of degenerating the Grassmannian of planes in an $n$ dimensional linear space into some toric variety described by a 3-valent tree with $n$ leaves. The dimensions we are looking for are invariant under this degeneration. I present properties of a semigroup of this toric variety. The dimensions can be presented in terms of characteristics of the semigroup.

	In the third chapter, I single out a special presentation of elements of the semigroup. I use this presentation and inclusion-exclusion principle to find a formula for the Poincare-Hilbert series. Additionally, I show a recursive formula for the Poincare-Hilbert series using independent combinatoric arguments. I present my recursive formula for a numerator of the Poincare-Hilbert series, but the technical proof is not attached in this version of the paper.

	In the fourth chapter, I show the method of calculating the Euler characteristic of invertible sheaves on a projective space with four blown-up points based on the Riemann-Roch formula and I compare it with the results from the previous chapter. \\

	I would like to thank prof. Jarosław Wisniewski for his indispensable help in writing this paper.

\begin{center}\end{center}
\section{Preliminaries}
In this paper we will assume that the reader knows the notation and basic facts from algebraic geometry, toric geometry and graph theory. The notions may be found in ($\cite{Ha}$), ($\cite{Cox}$, first chapter) and ($\cite{Sem}$) respectively. This chapter consists of material which is well known, so we will omit the majority of proofs. 

Trees (acyclic connected graphs) whose vertices have degree three will be called 3-valent trees. 
We will assume that every tree has a fixed embedding in a plane such that its edges do not intersect. Additionally, we will assume that leaves lie on a circle. This will be required to number the leaves.

Let $A$ be a ring with $\zz^n$-gradation where $A_0$ is a field and let $M$ be $A$-module with $\zz^n$-gradation. We define its Poincare-Hilbert series to be: 
\begin{eqnarray*} W(M) = \sum_{\lambda \in \zz^n} dim(M_{\lambda}) z^{\lambda} \in \zz[[z_1,z_2,\ldots,z_n]] \end{eqnarray*}
where $M_{\lambda}$ is a linear space of elements of $M$ which lie in the gradation $\lambda$. 

For elements of a lattice ($\lambda = (\lambda_1,\ldots, \lambda_n) \in \zz^n$) and for a ring of polynomials of $n$ variables $\zz[[z_1,\ldots,z_n]]$, we will let $z^{\lambda}$ denote $z_1^{\lambda_1}\ldots z_n^{\lambda_n}$.
\subsection{Toric geometry}
	The following notions, definitions and propositions can be found in the book $(\cite{Cox})$.
	The affine variety $T$ which is isomorphic to $\ntor$, will be called a torus. The group structure by coordinate-wise multiplication on $\ntor$ gives us the group structure on $T$. The morphism $\chi: T \to \ccs$, which is a homomorphism of algebraic groups, will be called a character of the torus $T$. 
	
	It holds that $Hom(\ntor, \ccs) \cong \zz^n$, where the character $\chi^m$ corresponds to the element $m = (m_1,\ldots,m_n) \in \zz^n$ so that
		\begin{eqnarray*}
			\chi^m(t_1,\ldots,t_n) = t_1^{m_1}\ldots t_n^{m_n}
		\end{eqnarray*}
	The group of characters of the torus $T$ will be denoted by $M_{T}$. The group of characters is a lattice. It is isomorphic to $\zz^n$, where $n$ is the dimension of the torus. 

	An affine irreducible variety $V$ will be called an affine toric variety if some Zariski open subset of it is isomorphic to a torus and the standard action of this torus on itself extends to the algebraic action on the variety $V$. 
	
	Let $T$ be a torus and let $M_{T}$ be its character lattice. The subset $\mathcal{A} = \{m_1,\ldots,m_s\} \subset M_{T}$ gives us the morphism $\Phi_{\mathcal{A}} : T \to \cc^s$ such that
	\begin{eqnarray*}
		\Phi_{\mathcal{A}}(t) = \big(\chi^{m_1}(t),\ldots,\chi^{m_s}(t)\big) \in \cc^s
	\end{eqnarray*}
	\begin{prop} The closure of the image $\Phi_{\mathcal{A}}(t)$ in $\cc^s$ is an affine toric variety with character lattice $\zz\mathcal{A}$. We will denote this variety by $Y_{\mathcal{A}}$. \label{tori_param}\end{prop}
	Each affine toric variety is isomorphic to $Y_{\mathcal{A}}$ for some $\mathcal{A}$.
	
	A semigroup $S$ will be called an affine semigroup if it is isomorphic to
		\begin{eqnarray*}
			\mathbb{N}\mathcal{A} = \{\sum_{m\in \mathcal{A}} a_m m \mbox{ } | \mbox{ } a_m \in \mathbb{N}\}
		\end{eqnarray*}
		for a finite subset $\mathcal{A}$ of some lattice $M$. 
			
	$\mathbb{N}\mathcal{A}$ will be called a semigroup of the affine toric variety $Y_{\mathcal{A}}$. 
	Let $\cc[S]$ be the algebra whose vector structure has a basis given by the elements of the semigroup $S$, with multiplication derived from the action of $S$. We will call it the algebra of the semigroup $S$. 
	
	\begin{prop} Let $M$ be a lattice and let $\mathcal{A}$ be a finite subset of $M$. Then $Spec(\cc[\mathbb{N} \mathcal{A}])$ is isomorphic to the affine toric variety $Y_{\mathcal{A}}$. The character lattice of the affine toric variety $Spec(\cc[\mathbb{N} \mathcal{A}])$ is $\zz \mathcal{A}$. \label{tori_semigroup} \end{prop}

\subsection{The Grassmannian of planes}
We will analyse the variety $\gg$ in the space $\mathbb{P}_{\cc}^{n(n-1)/2}$ with coordinates $x_{i,j}$ for $1 \leq i < j \leq n$, given by the equations
\begin{eqnarray}
	x_{i,j}x_{k,l} - x_{i,k}x_{j,l} + x_{i,l}x_{j,k} = 0 \hbox{ for } 1 \leq i < j < k < l \leq n
	\label{g_eq}
\end{eqnarray}

The algebra of this projective variety will be denoted $\cc[\gg] := \cc[x_{i,j}]/I$, where $I$ is the ideal generated by the proceeding polynomials. $\gg$ is a Grassmannian of planes in an $n$-dimensional linear space.

We introduce the action of the torus $\ntor$ on the variety $\gg$. Let $\sigma_t = (t_1,t_2,\ldots, t_n) \in \ntor$ act on coordinates in the following way:
\begin{eqnarray*}
	\sigma_t(x_{i,j}) = t_i t_j x_{i,j}.
\end{eqnarray*}
In other words, $t_k$ acts on $x_{i,j}$ with weight one exactly when $k=i$ or $k=j$. This action is well defined on the quotient ring $\cc[\gg]$ because $\left(\sigma_t \left(x_{i,j}x_{k,l} \right) = t_i t_j t_k t_l x_{i,j}x_{k,l}\right)$ and so generators of the ideal consist of monomials which belong to the same gradation. 
The action of the torus defines the natural $\zz^n$-gradation on $\cc[\gg]$. The element $f \in \cc[\gg]$ is in gradation $\lambda \in \zz^n$ if:
\begin{eqnarray*}
	\sigma_t(f) = t^{\lambda}f
\end{eqnarray*}

Our main goal is to calculate the Poincare-Hilbert series of the algebra $\cc[\gg]$ for the gradation described above.\\

The following theorem holds:
\begin{theorem} Let $A$ be a finitely generated $A_{0}$ algebra, where $A_{0}$ is a field, generated by elements $a_1, \ldots, a_k$ which belong to gradations $\lambda_1, \ldots, \lambda_k \in \zz^n$ respectively. Let $M$ be finitely generated over $A$. Then the Poincare-Hilbert series can be written in the following form:
\begin{eqnarray*}
	F(z_1,\ldots, z_n) / \prod_{i=1}^k (1 - z^{\lambda_i}) \mbox{ where } F \in \zz[z_1,\ldots,z_n]
\end{eqnarray*} 
\label{Grassmanian_Hilbert}
\end{theorem}
We should note that each rational function with an invertible constant term in its denominator defines a unique formal power series.

The proof of this theorem repeats the argument from ($\cite{Ati}$, 11.1) applied to a multidimensional gradation.

The coordinate $x_{i,j}$ of the variety $\gg$ belongs to the gradation $t_it_j$. Therefore, the Poincare-Hilbert series of the Grassmannian can be written in the following form:
	\begin{eqnarray*}F(z_1,\ldots, z_n) / \prod_{1 \leq i < j \leq n} (1 - z_iz_j) \mbox{ where } F \in \zz[z_1,\ldots,z_n] \end{eqnarray*}

\subsection{A variety described by a tree}
Let us recall that we assume that each tree has a fixed embedding in a plane in which edges do not intersect and leaves lie on a circle.

A 3-valent tree with $n$ leaves (for $n>2$) has $2n-2$ vertices, $2n-3$ edges and contains a vertex which is a neighbour of two leaves. We number leaves (anticlockwise around the tree) and edges with consecutive natural numbers starting from $1$. Between each pair of leaves there exists a unique path. Let $p_{i,j}$ be the set of numbers labelling the edges of the path connecting $i$ and $j$. 

Let $\tt$ be some 3-valent tree with $n$ leaves. Let us construct a toric variety described by this tree. Consider the torus $(\ccs)^{2n-3}$ spanned by the edges of the tree (let $y_1,\ldots, y_{2n-3}$ be the coordinates corresponding to subsequent edges) and the subset of its characters $\mathcal{A_{\tt}} = \{\chi_{i,j} \mbox{ } | \mbox{ } 1 \leq i < j < n\}$, where
\begin{eqnarray*}
	\chi_{i,j}(y_1,\ldots,y_{2n-3}) = \prod_{e \in p_{i,j}} y_e    
\end{eqnarray*}
In other words, the character $\chi_{i,j}$ returns the product of coordinates on the path connecting $i$-th and $j$-th leaves.

The set of characters $\mathcal{A_{\tt}}$ gives us the toric variety $Y_{\mathcal{A_{\tt}}}$. We will denote it by $\Xaff$ (because of its importance in the paper). It is the closure in $\cc^{n(n-1)/2}$ of the map $\Phi_{\mathcal{A_{\tt}}} : (\ccs)^{2n-3} \to (\ccs)^{n(n-1)/2}$ from the space spanned on edges to the space spanned on pairs of leaves which maps pairs of leaves to the product of coordinates on the path connecting those leaves.\\

	The variety $\Xaff$ is a cone - it induces a projective variety in a space $\pp_\cc^{n(n-1)/2-1}$.
The variety constructed in this way will be called a variety described by a tree and will be denoted by $\Xtt$. Let us note, that $\Xaff$ is an affine cone over $\Xtt$.\\

\begin{defi} Let $i<j$ be some leaves of the tree $\tt$. We define an element $w_{i,j}$ in the lattice $Z^{2n-3}$ associated to the edges of $T$ which has ones on coordinates corresponding to edges on shortest path between $i$-th and $j$-th leaves (i.e. on positions $p_{i,j}$) and zeroes on others. \end{defi}
By $w_{i,j}$ we will also mean the path between $i$-th and $j$-th leaves.

Let $\zz^{2n-3}$ be a standard character lattice of a torus $(\ccs)^{2n-3}$. Under this identification, vector $w_{i,j} \in \zz^{2n-3}$ corresponds to the character $\chi_{i,j} \in \mathcal{A_{\tt}}$. The fact ($\ref{tori_param}$) implies that the character lattice $\Xaff$ is generated by paths, i.e.:
	\begin{eqnarray*} 
		\zz \mathcal{A_{\tt}} \cong \{\sum_{i,j} a_{i,j} w_{i,j} \mbox { } | \mbox { } a_{i,j} \in \zz\}
	\end{eqnarray*}
In informal words: we can perceive elements of the $\Xaff$ lattice as such assignments of integer numbers to edges which come from adding and subtracting paths (as vectors).\\

We will analyse the semigroup $\S$ describing the toric variety $\Xaff$ in order to find our Hilbert series: 
	\begin{eqnarray*} 
		\S = \nn \mathcal{A_{\tt}} \cong \{\sum_{i,j} a_{i,j} w_{i,j} \mbox { } | \mbox { } a_{i,j} \in \nn\}
	\end{eqnarray*}
The semigroup $\S$ is the subset of the character lattice of the variety $\Xaff$ and is generated by positive (i.e. with positive coefficients) linear combinations of paths. In informal words: we can perceive elements of the semigroup $\S$ as such assignments of integer nonnegative numbers to edges, which come from adding paths (as vectors). The operation in this semigroup is just adding numbers on corresponding edges.\\

Let us recall that ($\ref{tori_semigroup}$) implies that the variety $\Xaff$ is isomorphic to the variety $Spec(\cc[\S])$. Therefore their algebras are also isomorphic.
\begin{eqnarray*} \cc[\Xaff] \cong \cc[\S] \end{eqnarray*}

Obviously, elements of the semigroup do not decompose uniquely into positive linear combinations of paths. E.g. for (the circled numbers are the values of the edges and the other numbers label the edges):
\begin{figure}[here] \centering \includegraphics[width=10cm]{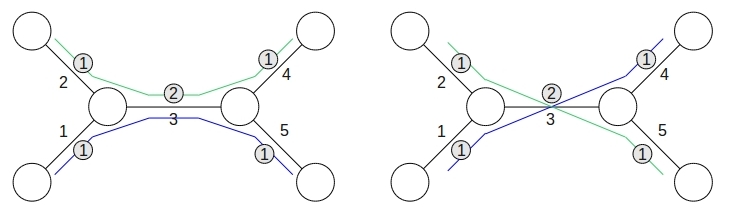} \caption{Examples of paths in the tree} \label{fig:graphic_tree} \end{figure}
\begin{center}\end{center}
we have:
	\begin{itemize}
		\item for the first tree: $[1,1,2,1,1] = [1,0,1,0,1] + [0,1,1,1,0]$
		\item for the second tree: $[1,1,2,1,1] = [1,0,1,1,0] + [0,1,1,0,1]$
	\end{itemize}

We introduce the following natural metric on vertices of a tree. The distance between two vertices is the number of edges in the shortest path which connects those vertices. We will denote the distance between the vertices $i$ and $j$ by $d(i,j)$.

Let us observe, that in a 3-valent tree paths have common vertex if and only if they have a common edge.

We prove the following lemma, because we will use similar ideas later in the paper.
\begin{lem}
	\label{paths_intersection}
	Let $1 \leq i < j < k < l \leq n$. Then either $w_{i,j}$ intersects $w_{k,l}$, or $w_{i,l}$ intersects $w_{j,k}$. $w_{i,k}$ always intersects $w_{j,l}$.
\end{lem}
\begin{proof}
	Let $v$ be the last common vertex of $w_{i,j}, w_{i,k}$ and $w_{i,l}$. The vertex $i$ is a leaf, so $v \neq i$. The vertex $v$ splits the tree into three subtrees: $\tt_1, \tt_2, \tt_3$ (numbered anticlockwise in such a way that $i$ belongs to $\tt_1$). Due to the order of leaves, $j$ must belong to $\tt_2$ and $l$ to $\tt_3$. If: 
	\begin{enumerate}
		\item $k$ lies in $\tt_2$, then path $w_{i,l}$ (disjoint with $\tt_2$) doesn't intersect $w_{j,k}$ (belonging to $\tt_2$), but $w_{i,j}$ intersects $w_{k,l}$ (the edge from $v$ to $\tt_2$ lies in the intersection).
		\item $k$ lies in $\tt_3$, then path $w_{i,j}$ (disjoint with $\tt_3$) doesn't intersect $w_{k,l}$ (belonging to $\tt_3$), but $w_{i,l}$ intersects $w_{j,k}$ (the edge from $v$ to $\tt_3$ lies in the intersection).
	\end{enumerate}
	Paths $w_{i,k}$ and $w_{j,l}$ contain the vertex $v$, so they intersect.

\begin{figure}[here] \centering \includegraphics[width=12cm]{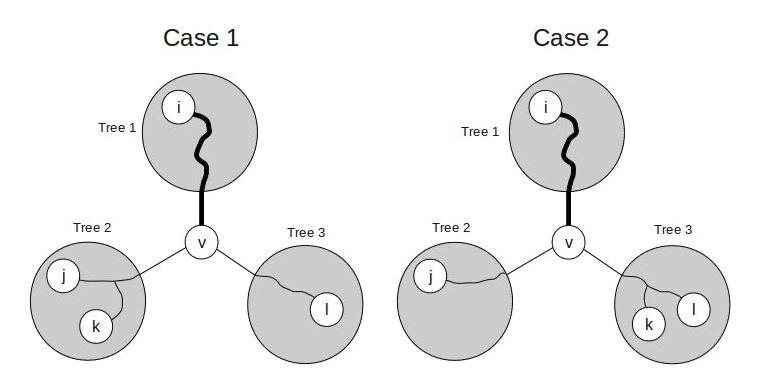} \caption{Two possibilities of intersection} \label{fig:graphics_tree_intersection} \end{figure}
\end{proof}

The lemma implies the following corollary: 
\begin{cor}
	The following inequalities hold:
	\begin{enumerate}
		\item $d(i,l) + d(j,k) < d(i,j)+d(k,l) = d(i,k) + d(j,l)$, if $w_{i,j}$ intersects $w_{k,l}$
		\item $d(i,j) + d(k,l) < d(i,l)+d(j,k) = d(i,k) + d(j,l)$, if $w_{i,l}$ intersects $w_{j,k}$ 
	\end{enumerate}
	\label{cor_inters}
\end{cor}

\begin{theorem}
One of the following polynomials lie in the ideal of $\Xaff$ for every four numbers $1 \leq i < j < k < l \leq n$:
\begin{enumerate}
	\item $W_1(i,j,k,l) = x_{i,j}x_{k,l} - x_{i,k}x_{j,l}$, if $w_{i,j}$ intersects $w_{k,l}$
	\item $W_2(i,j,k,l) = x_{i,l}x_{j,k} - x_{i,k}x_{j,l}$, if $w_{i,l}$ intersects $w_{j,k}$
\end{enumerate}
The polynomials given above generate the ideal $\Xaff$.
\end{theorem}

We introduce the action of the torus $\ntor$ on the variety $\Xtt$ in the same way as on the variety $\gg$. Let $\sigma_t = (t_1,t_2,\ldots, t_n) \in \ntor$ act on coordinates in the following way:
     \begin{eqnarray*}
         \sigma_t(x_{i,j}) = t_i t_j x_{i,j}.
     \end{eqnarray*}
The action is well defined on the quotient ring $\cc[\Xtt]$, because generators of the ideal consist of monomials which belong to the same gradation.\\

The action of the torus defines a natural $\zz^n$-gradation on $\Xtt$ in the same way as on $\gg$. An element $f \in \cc[\Xtt]$ belongs to a gradation $\lambda \in \zz^n$ if: 
	\begin{eqnarray*}
		\sigma_t(f) = t^{\lambda}f
	\end{eqnarray*}
Note that this action of the torus is not faithful. 

The gradation of the algebra can be expressed in the combinatorial language of toric varieties. 
Let $\pi : \zz^{2n-3} \to \zz^{n}$ be a projection from the space spanned by the edges of the tree to the space spanned by the edges which are incident to leaves. In other words, $\pi$ forgets about inner coordinates of the tree.

Recall that elements of $\S$ may be treated as monomials in $\cc[\Xaff]$. 
\begin{lem} \label{grad} An element $a \in \S$, treated as monomial in $\cc[\S]$, lies in the gradation $\pi(a)$. \end{lem}
	Informally, it must be shown that the gradation of the semigroup is determined by the values corresponding to the edges which are incident to leaves. To a gradation $\lambda \in \zz^n$ belong elements of the semigroup having value $\lambda_i$ on the edge incident to the $i$th leaf. 
	
	For example element $[1, 1, 2, 1, 1]$ lies in gradation $[1, 1, 1, 1]$ (see Figure ($\ref{fig:graphic_tree}$)).

\subsection{Degeneration of the Grassmannian to a variety described by a tree}
A proof of the following theorem (stated in a much more general situation) may be found in ($\cite{Deg}$, theorems 10.6, 5.2, 5.3). We present a simple proof for convenience of the reader. 
\begin{theorem} $\gg$ degenerates to $\Xtt$ for every tree $\tt$. That is, there exists a variety in $\pp_{\cc}^{n(n-1)/2-1} \times \cc$ such that in the projection to $\cc$, the fiber over an arbitrary nonzero point is equal to $\gg$ and the fiber over zero is $\Xtt$. \label{degeneration} \end{theorem}
\begin{proof}
	Let us recall that $d(i,j)$ is a length of $w_{i,j}$. Consider the following action of the torus $\cc^{*}$ on $\pp_{\cc}^{n(n-1)/2-1}$:
	\begin{eqnarray*}
		t(x_{i,j}) = t^{-d(i,j)}x_{i,j}
	\end{eqnarray*}
	We define on $\pp_{\cc}^{n(n-1)/2-1} \times \cc^*$ a variety $V'$ given by equations:
	\begin{eqnarray}
		t(x_{i,j})t(x_{k,l}) - t(x_{i,k})t(x_{j,l}) + t(x_{i,l})t(x_{j,k}) = 0 \hbox{ for } 1 \leq i < j < k < l \leq n
		\label{big_var_eq}
	\end{eqnarray}
	The map $(x_{i,j}, t) \to (t(x_{i,j}), t)$ induces an isomorphism over $\cc^*$ from $V'$ to $\gg \times \cc^*$. Therefore the fiber over an arbitrary point of the projection from $v'$ to $\cc^*$ is isomorphic to $\gg$.\\

	We would like to extend $V'$ to $\pp_{\cc}^{n(n-1)/2-1} \times \cc$. Let $V \subset \pp_{\cc}^{n(n-1)/2-1} \times \cc$ be given by the following equations:
	\begin{eqnarray*}
		&& x_{i,j}x_{k,l} - x_{i,k}x_{j,l} + t^{d(i,k)+d(j,l)-d(i,l)-d(j,k)}x_{i,l}x_{j,k} = 0 \mbox { if } w_{i,j} \mbox { intersects } w_{k,l} \\
		&& t^{d(i,l)+d(j,k)-d(i,j)-d(k,l)}x_{i,j}x_{k,l} - x_{i,k}x_{j,l} + x_{i,l}x_{j,k} = 0 \mbox { if } w_{i,l} \mbox { intersects } w_{j,k}
	\end{eqnarray*}
	To receive these equations we multiplied ($\ref{big_var_eq}$) by the required power of $t$ (we are using ($\ref{cor_inters}$)).\\ 
	
	We see that $V = V'$ for $t \neq 0$. Additionally, the fiber over the zero of the projection from $V$ to $\cc$ is isomorphic to $\Xtt$ (let us substitute $t=0$ in the equations mentioned above - the corollary ($\ref{cor_inters}$) implies that the exponent of the power of $t$ is positive). $V$ is the variety describing degeneration which we were looking for.
\end{proof}

\begin{theorem} The Poincare-Hilbert series for $\gg$ and $\Xtt$ (with respect to the action of a torus $\ntor$ described above) are the same.  
\end{theorem}
The proof of this theorem is identical to ($\cite{Wis}$, 2.35).
\begin{proof}
Let us consider the standard action of the torus $\cc^*$ (with a weight one on each variable) on the varieties $\gg$ and $\Xtt$. The action extends uniquely to $V$ (notions come from ($\ref{degeneration}$)). The projection of $\cc$-varieties from $V$ to $\cc$ is flat. The projective variety $V$ is parametrized by $\cc$. The theorem ($\cite{Ha}$, III.9.9) about preserving Hilbert series for flat families implies that the Hilbert series (for the action of the torus $\cc^*$) of fibers of $V$ over $\cc$ are the same.

	Note that the dimension of the gradation $t \in \zz$ for $\gg$ (or $\Xtt$) is equal to $h^0(O_{\gg}(t))$ (respectively $h^0(O_{\Xtt}(t))$). It implies that the Hilbert series of fibers are the same and we can apply Grauert theorem ($\cite{Ha}$, III.12.9) for a zeroth derived functor. The sheaf $\oo(t)$ pushed on $\cc$ by projection from $V$ is locally constant and stalks of its zero cohomologies correspond to zero cohomologies of the sheaf $\oo(t)$ on fibers. Therefore zero cohomologies of fibers are locally the same and the eigenspaces of the action of the torus $\ntor$ (we use the fact that the action agrees with the action of $\cc^*$) on fibers are also locally the same. 
\end{proof}

\begin{cor} The Poincare-Hilbert series for $\Xtt$ is independent of the choice of 3-valent tree $\tt$ and is symmetric with respect to the variables $z_1,\ldots,z_n$. \label{hilbert_independence} \end{cor}
\begin{proof} The independence follows straightforwardly from the theorem above. The series is symmetric, since it does not change if edges of $\tt$ are relabeled. \end{proof}

\section{The methods of calculating the Hilbert series}

\subsection{A semigroup of a variety described by a tree} 
We will consider only paths which are the shortest paths between leaves (of the form $w_{i,j}$ for leaves $i<j$).

We introduce the notion of an ``ordered intersection of paths'' which will be used to single out a particular type of decomposition into sums of paths of $\S$ elements.\\

In this section all pairs will be ordered, i.e. for $(i,j)$ it holds that $i<j$. We compare two ordered pairs lexicographically. The ordering on pairs induces the ordering on paths:
	\begin{eqnarray*} w_{i,j} < w_{k,l} \mbox { iff } (i,j) < (k,l) \mbox { where } i<j \mbox { and } k<l \end{eqnarray*}

Subsequently, it induces the lexicographic ordering on ordered pairs of paths. For paths $w_1 < w_2$ and $w'_1 < w'_2$
	\begin{eqnarray*} (w_1,w_2) < (w'_1, w'_2) \mbox { iff } w_1 < w'_1 \mbox { or } w_1 = w'_1 \mbox { and } w_2 < w'_2 \end{eqnarray*}

We will introduce a notion of duality. Consider an ordered pair of intersecting paths $(w_{a,b}, w_{c,d})$ without common endpoints, where $a<b$, $c<d$ and $w_{a,b} < w_{c,d}$. We can divide four leaves $a,b,c,d$ into two pairs in exactly three ways. For each division we connect each pair of leaves by a path. The lemma ($\ref{paths_intersection}$) implies that in two out of three cases, the paths will intersect (exactly for a pair $(w_{a,b},w_{c,d})$ and for some ordered pair $(w_{a',b'},w_{c',d'})$), and in one case they won't. We will call a pair of paths $(w_{a',b'},w_{c',d'})$ a dual to a pair $(w_{a,b}, w_{c,d})$. A dual pair is constructed from a pair of paths by exchanging two endpoints in a such a way that the paths still intersect. Note that $w_{a,b} + w_{c,d} = w_{a',b'} + w_{c',d'}$ (see Figure $\ref{fig:graphics_tree_intersection}$).

In other words a dual to our pair $(w_{a,b},w_{c,d})$ is an ordered pair of intersecting paths $(w_{a',b'},w_{c',d'})$ ($a'<b'$, $c'<d'$, $w_{a',b'} < w_{c',d'}$) such that $\{a',b',c',d'\} = \{a,b,c,d\}$. 

Observe that a dual pair is uniquely determined and the notion of duality is symmetric.\\

\begin{defi} Let us consider two intersecting paths $w_1 \in \zz^{2n-3}$ and $w_2 \in \zz^{2n-3}$ without common endpoints. Let $(w'_1, w'_2)$ be a dual pair to $(w_1,w_2)$. We say that $w_1$ intersects $w_2$ in an ordered way if $(w_1,w_2) < (w'_1,w'_2)$. Otherwise we say that they intersect in an unordered way. \end{defi}
It is clear from the definition that the paths $w_1,w_2$ intersect in an ordered way if and only if the paths $w'_1,w'_2$ which are dual to $(w_1,w_2)$ intersect in an unordered way.

In the case when two paths have the same endpoint, which implies that they intersect, we will say that the paths intersect in an ordered way.

For instance, in Figure ($\ref{fig:graphic_tree}$) paths in the first tree intersect in an unordered way. After exchanging the endpoints of this pair of paths, we get the situation in the second tree, in which paths intersect in an ordered way. 

\begin{prop}
	Each $x \in \S \subset \zz^{2n-3}$ decomposes into a sum of paths $w_{i,j}$ such that no two paths in this decomposition intersect in an unordered way.  \label{decomposition_exist}
\end{prop}
\begin{proof}
	Let us choose the smallest lexicographical decomposition of $x$ into a sum of paths - we treat a sum of paths as ordered sequence of summands and compare lexicographically. We order paths in the way stated at the beginning of the section. 

	Suppose that two paths $w$ and $w'$ in this decomposition intersect in an unordered way. Let $(v,v')$ be the dual pair (constructed from $w,w'$ by replacing endpoints). After replacing the paths $w$ and $w'$ by $v$ and $v'$ in the decomposition of $x$, we still obtain a decomposition of $x$ since $w+w'$ = $v+v'$. This new decomposition is lexicographically smaller since the definition of an unordered intersection implies that: $(v,v') < (w,w')$. 
\end{proof}

\begin{prop}
	\label{decomposition}
	Let $x \in \S \subset \zz^{2n-3}$. Then $x$ has a unique decomposition into a sum of paths $w_{i,j}$ such that no two paths in the decomposition intersect in an unordered way.
\end{prop}
\begin{proof}
	If $n=3$, no pairs of paths intersect in an unordered way. Let
	\begin{eqnarray*}
		x = [x_1, x_2, x_3] = a_{1,2}w_{1,2} + a_{1,3}w_{1,3} + a_{2,3}w_{2,3}.
	\end{eqnarray*}
	where $x_i$ is the value on the $i$th edge. Then $a_{1,2},a_{2,3}, a_{1,3}$ are unique solutions of the system of three equations with three variables.
	\begin{eqnarray*}
		a_{1,2} = \frac{1}{2}(x_1 + x_2 - x_3) \hspace{0.5cm} a_{1,3} = \frac{1}{2}(x_1 + x_3 - x_2) \hspace{0.5cm} a_{2,3} = \frac{1}{2}(x_2 + x_3 - x_1) 
	\end{eqnarray*}

	Therefore the decomposition is unique.\\

	Let us assume that the decomposition is unique for trees having $n-1$ leaves. Let us choose a vertex $v$ incident to two leaves $l_1$ and $l_2$, $l_1 < l_2$.  Let $\tt'$ be the tree constructed from $\tt$ by erasing $l_1$ and $l_2$. Notice that $v$ is a leaf in $\tt'$. Let $p_n$ be a projection from the space $\zz^{2n-3}$ of edges of the tree $\tt$ to the space $\zz^{2n-5}$ of edges of the tree $\tt'$. Observe that if $x \in S(\tt)$ then $p_n(x) \in S(\tt')$. 
	
	Let $x = \sum_{i,j} a_{i,j} w_{i,j}$ be a decomposition into paths in which every two paths intersect in an ordered way ($a_{i,j} \in \nn$). Then $p_n(x) = \sum a_{i,j} p_n(w_{i,j})$. Note that $\{ p_n(w_{i,j})$ for $a_{i,j} \neq 0 \}$ is a set of paths in $\tt'$ in which every two paths intersect in an ordered way. By induction this decomposition of $p_n(x)$ is unique.

	Consequently, the paths in the decomposition of $x$ which are disjoint to $l_1$ and $l_2$ are uniquely determined. Paths passing through $v$ (except $l_1$ and $l_2$) are nearly uniquely determined, except that they may end in either $l_1$ or $l_2$.  
	
	Let $x_{l_1},x_{l_2}, x_{v}$ be the values of $x$ on three edges coming out of $v$, where $x_{l_1}$, $x_{l_2}$ are the values on the edges which are incident to $l_1$ and $l_2$ respectively. Then the number of edges between $l_1$ and $l_2$ is $(x_{l_1}+x_{l_2} - x_{v})/2$ similarly to the case $n=3$. So the number of paths leaving from $l_1$ and not entering $l_2$ (and so entering $l_3$) is $y_{l_1} := x_{l_1} -(x_{l_1}+x_{l_2} - x_{v})/2$. Analogously, the number of paths from $l_2$ which do not enter $l_1$ is $y_{l_2} := x_{l_2} -(x_{l_1}+x_{l_2} - x_{v})/2$. Clearly $y_{l_1} + y_{l_2} = x_v$.
	
	Paths starting in $l_1$ must end in leaves with smaller (or equal) numbers than paths starting in $l_2$, since otherwise we would have a pair of paths which intersect in an unordered way. Therefore paths which pass through $v$ are uniquely determined - under the lexicographic ordering the first $y_{l_1}$ paths which end in $v$ in the decomposition of $p_n(x)$ must be extended to paths which end in $l_1$ and the other $y_{l_2}$ must be extended to paths which end in $l_2$. 
\end{proof}

The two propositions mentioned above imply:
\begin{cor}
Elements of $\S$ are in bijection with such sums of paths in which no two paths intersect in an unordered way. \label{s_bijection}
\end{cor}

We have shown that the Hilbert series is independent of the choice of the tree $\tt$ (corollary $(\ref{hilbert_independence})$). We will consider trees $\tt_{n+1}$ of the following form:
\begin{figure}[here] \centering \includegraphics[width=10cm]{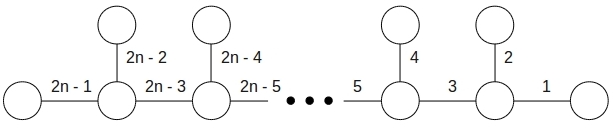} \caption{Tree $\tt_{n+1}$} \label{fig:simple_tree} \end{figure}\\
The leaves are numbered from $1$ to $n+1$ anticlockwise - from right to left. 

\subsection{A combinatorial interpretation of dimensions of torus action's eigenspaces} 
Observe that $dim(\cc[X(\tt_n)]_{\lambda})$ is equal to the number of elements of $\Sn$ which lie in the gradation $\lambda$, since $\cc[\Sn] \cong \cc[X(\tt_n)]$ and the algebra $\cc[\Sn]$, considered as a group, is free.\\ 

 Let $r_{i,j} \in \zz^n$ denote the vector with ones in the $i$th and $j$th coordinates ($i < j$) and zero everywhere else.
We will say that a vector $r_{i,j}$ (respectively pair $(i,j)$) embraces $r_{i',j'}$ (pair $(i',j')$), if $i < i' < j' < j$.

\begin{theorem} $dim(\cc[X(\tt_n)]_{\lambda})$ is equal to the number of decompositions of $\lambda \in \zz^n$ into a sum of vectors $r_{i,j}$ in which no term embraces any other term. \label{nominator_combi} \end{theorem}
For example, for $n=4$ and $\lambda = [1,1,1,1]$ we have the following decompositions:  
	\begin{eqnarray*}
		\lambda = [1,1,0,0] + [0,0,1,1] = [1,0,1,0] + [0,1,0,1] = [1,0,0,1] + [0,1,1,0]
	\end{eqnarray*}

The last decomposition is "invalid", because $[1,0,0,1]$ embraces $[0,1,1,0]$. This theorem implies that $dim(\cc[X(\tt_4)])_{[1,1,1,1]} = 2$.
\begin{proof}
Firstly, note that in the tree $\tt_n$ two paths $w_{i,j}$ and $w_{i',j'}$ intersect in an unordered way if and only if one pair embraces other one, that is, either $(i,j)$ embraces $(i',j')$ or $(i',j')$ embraces $(i,j)$.

Let $x \in \Sn$ lie in the gradation $\lambda$ and be equal to $\sum_{i,j} a_{i,j} w_{i,j}$, where $a_{i,j} \in \nn$ Then
	\begin{eqnarray*}
		\lambda = \sum_{i,j} a_{i,j} r_{i,j}
	\end{eqnarray*}

Proposition ($\ref{s_bijection}$) shows that the elements of $\Sn$ are in bijection with sums of paths, for which no path intersects another path in an unordered way.

The remarks above show that the elements of $\Sn$ which lie in a gradation $\lambda$ are in bijection with decompositions of $\lambda$ into sums of vectors $r_{i,j}$ such that no vector embraces any other.
\end{proof}

\subsection{Formulas for Hilbert series}
	Let us define a function $Multi$ from sequences of ordered pairs of integers $(i,j)$ s.t. $1 \leq i < j \leq n$ to $\zz[z_1, \ldots, z_n]$.
	\begin{eqnarray*}
		Multi((i_1,j_1),(i_2,j_2), \ldots, (i_k,j_k)) = \prod_{\mbox{ over distinct pairs } (i_l,j_l)} z_{i_l}z_{j_l}
	\end{eqnarray*}
	For example $Multi((1,2),(1,3), (1,2),(2,4)) = z_1z_2 \cdot z_1z_3 \cdot z_2z_4 = z_1^2z_2^2z_3z_4$.\\

	We also define a function $Sum$ from sequences of ordered pairs of integers $(i,j)$ s.t. $1 \leq i < j \leq n$ to $\zz^n$.
	\begin{eqnarray*}
		Sum((i_1,j_1),(i_2,j_2), \ldots, (i_k,j_k)) = \sum_{\mbox{ over distinct pairs } (i_l,j_l)} r_{i_l,j_l}
	\end{eqnarray*}
	For example $Sum((1,2),(1,3), (1,2),(2,4)) = r_{1,2} + r_{1,3} + r_{2,4} = [1,1,0,0] + [1,0,1,0] + [0,1,0,1] = [2,2,1,1]$.\\

	It holds that
	\begin{eqnarray}
		Multi((i_1,j_1),(i_2,j_2), \ldots, (i_k,j_k)) = z^{Sum((i_1,j_1),(i_2,j_2), \ldots, (i_k,j_k))}
		\label{MultiSum}
	\end{eqnarray}\\

	Let $Exc = \{ ((i,j), (i',j')) \mbox{ } | \mbox{ } 1 \leq i < i' < j' < j \leq n;\mbox{ } i,j,i',j' \in \nn \}$. In other words, $Exc$ is a set of 2-tuples of pairs such that the first pair embraces the second one. We introduce a natural lexicographic ordering on $Exc$. Let $e^1 = (i,j)$ and $e^2 = (i',j')$ for $e = ((i,j), (i',j'))$. 
	
	We define functions $\widetilde{Multi}$ from sequences of elements of Exc to $\zz[z_1,\ldots,z_n]$ and $\widetilde{Sum}$ from sequences of elements of Exc to $\zz^n$ in the following way:
	\begin{eqnarray*} 
		& & \widetilde{Multi}(e_1,\ldots,e_n) := Multi(e_1^1, e_1^2, \ldots, e_k^1, e_k^2) \mbox { where } e_1,\ldots,e_n \in Exc  \\
		& & \widetilde{Sum}(e_1,\ldots,e_n) := Sum(e_1^1, e_1^2, \ldots, e_k^1, e_k^2) \mbox { where } e_1,\ldots,e_n \in Exc 
	\end{eqnarray*}

\begin{theorem} The Hilbert-Poincare series $W_n$ is equal to
	\begin{eqnarray*} \big(1 + \sum_{k=1}^n (-1)^k\sum_{\substack{e_1 < \ldots < e_k \\ e_1,\ldots, e_k \in Exc} } \widetilde{Multi}(e_1, \ldots, e_k)\big)  / \prod_{1 \leq i < j \leq n} (1 - z_iz_j) 
	\end{eqnarray*}
\end{theorem}
\begin{proof}
	Let $\Omega^{\lambda}$ denote the set of all decompositions of $\lambda \in \zz^n$ into sums of vectors $r_{i,j}$. Let $A^{\lambda}$ denote the set of all decompositions for which no summand embraces other one. Let $\Omega^{\lambda}_{(a,b),(a',b')}$ denote the set of all decompositions in which both summands $r_{a,b}$ and $r_{a',b'}$ occur.

	Theorem ($\ref{nominator_combi}$) implies that $|A^{\lambda}| = dim(\cc[X(\tt_n)]_{\lambda})$.
	By definition
	\begin{eqnarray*}
		A^{\lambda} = \Omega^{\lambda} \backslash \bigcup_{e \in Exc} \Omega^{\lambda}_{e}
	\end{eqnarray*}

	By inclusion-exclusion formula we have that
	\begin{eqnarray}
		|A^{\lambda}| = |\Omega^{\lambda}| - \sum_{k=1}^n (-1)^{k-1} \sum_{\substack{e_1 < \ldots < e_k \\ e_1,\ldots, e_k \in Exc} } |\bigcap_{l=1}^k \Omega^{\lambda}_{e_l}|
		\label{in_and_out_formula}
	\end{eqnarray}

	Let
	\begin{eqnarray*}
		W(\Omega) = \sum_{\lambda \in \zz^n} |\Omega^{\lambda}| z^{\lambda}
	\end{eqnarray*}

	Note that $\bigcap_{l=1}^k \Omega^{\lambda}_{e_l}$ consists of exactly those decompositions which contain summands $r_{e_1^1}$, $r_{e_1^2}$, $\ldots$, $r_{e_k^1}$, $r_{e_k^2}$ (note that subsequent occurrences of the pairs should be omitted - each pair is considered at most once). The element, in decomposition of which each of this summands is contained exactly ones, is equal to $\widetilde{Sum}(e_1,\ldots, e_k)$. There is a natural bijection between decompositions of $\lambda$ containing summands mentioned above and arbitrary decompositions of $\lambda - \widetilde{Sum}(e_1,\ldots,e_k)$. Therefore
	\begin{eqnarray*} |\bigcap_{l=1}^k \Omega^{\lambda}_{e_l}| = |\Omega^{\lambda - \widetilde{Sum}(e_1,\ldots,e_k)}| \end{eqnarray*}

	For fixed $e_1,\ldots,e_k \in Exc$ we have that
	\begin{eqnarray*}
		\sum_{\lambda \in \zz^n} |\bigcap_{l=1}^k \Omega^{\lambda}_{e_l}| z^{\lambda} & = & \sum_{\lambda \in \zz^n} |\Omega^{\lambda - \widetilde{Sum}(e_1,\ldots,e_k)}| z^{\lambda}  =  \\
		& & \mbox { take } \lambda' = \lambda - \widetilde{Sum}(e_1,\ldots,e_k) \\
		& = & \sum_{\lambda' \in \zz^n} |\Omega^{\lambda'}| z^{\lambda'} z^{\widetilde{Sum}(e_1,\ldots,e_k)} \mathop{=}^{(\ref{MultiSum})} \\
		& = & \sum_{\lambda' \in \zz^n} |\Omega^{\lambda'}| z^{\lambda'} \widetilde{Multi}(e_1,\ldots,e_k) = \\
		& = & W(\Omega) \widetilde{Multi}(e_1,\ldots,e_k) 
	\end{eqnarray*}

	Let us multiply both sides of equality ($\ref{in_and_out_formula}$) by $z^{\lambda}$ and sum over $\lambda \in \zz^n$. We obtain:
	\begin{eqnarray*}
		W_n = W(\Omega) + \sum_{k=1}^n (-1)^{k} \sum_{\substack{e_1 < \ldots < e_k \\ e_1,\ldots, e_k \in Exc} } W(\Omega) \widetilde{Multi}(e_1,\ldots,e_l)
	\end{eqnarray*}

	The conclusion of the theorem follows from the following lemma:
	\begin{lem}
		\begin{eqnarray*} W(\Omega) = 1 / \prod_{1 \leq i < j \leq n} (1 - z_iz_j) \end{eqnarray*}
	\end{lem}
	\begin{proof}
		Observe that
		\begin{eqnarray*}
				1 / \prod_{1 \leq i < j \leq n} (1 - z_iz_j) = \prod_{1 \leq i < j \leq n} (1 + z_iz_j + (z_iz_j)^2 + \cdots)
		\end{eqnarray*}
		After expanding this product, we see that the coefficient of the $z^{\lambda}$ term is equal to the number of decompositions of $\lambda$ into the sum of $r_{i,j}$, which proves the lemma. 
	\end{proof}
\end{proof}

Note that if we define $Exc$ as a set of pairs of paths intersecting in an unordered way, then the theorem will be true for an arbitrary tree with an analogous proof.

\subsection{A recursive formula for Hilbert series}
\begin{figure}[here] \centering \includegraphics[width=10cm]{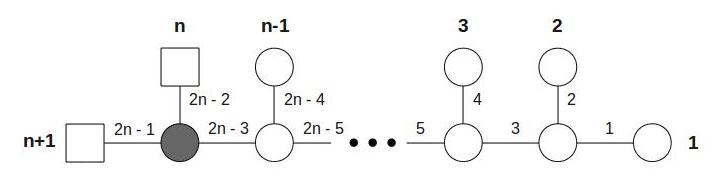} \caption{Tree $\tt_{n+1}$} \label{fig:simple_tree2} \end{figure}

We will try to calculate the semigroup $\Snp$ from $\Sn$ recursively.

Let us denote by $p_{n+1}$ the projection $\zz^{2n-1} \to \zz^{2n-3}$
\begin{eqnarray*} p_{n+1}(s_1,\ldots, s_{2n-3},s_{2n-2},s_{2n-1}) = (s_1,\ldots,s_{2n-3}) \mbox { where } s_i \in \zz \end{eqnarray*}

$\tt_{n+1}$ can be constructed from $\tt_n$ by adding two vertices (squares) and corresponding incident edges to the last leaf of $\tt_n$ (filled circle). The last leaf of $\tt_n$ becomes an interior vertex. 

We treat $p_{n+1}$ as a restriction of a lattice $\zz^{2n-1}$ spanned by edges of $\tt_{n+1}$ to a lattice $\zz^{2n-3}$ spanned by edges of $\tt_{n}$, in which the values of the last two edges (those with biggest numbers - see Figure ($\ref{fig:simple_tree2}$)) of the tree $\tt_{n+1}$ in lattice $\zz^{2n-1}$ are ``forgotten''. Recall that $\Snp \subset \zz_{2n-1}$ and $\Sn \subset \zz_{2n-3}$.

Let $w_{i,j} \in \zz^{2n-1}$ be the shortest path between $i$th and $j$th leaves in $\tt_{n+1}$ and let $\tilde{w}_{i,j} \in \zz^{2n-3}$ be the shortest path between $i$th and $j$th leaves in $\tt_{n}$. \\

Under the restriction $p_{n+1}$,
		\begin{itemize}
			\item the path between the last two leaves of $\tt_{n+1}$ ``becomes empty'', i.e.
				\begin{eqnarray*} p_{n+1}(w_{n,n+1}) = 0 \end{eqnarray*}
			\item other paths which start in the last two leaves shrink by one and subsequently they start at the last leaf of $\tt_n$, i.e.
				\begin{eqnarray*} 
					& & p_{n+1}(w_{i,n}) = \tilde{w}_{i,n} \mbox { for } i \neq n+1 \\
					& & p_{n+1}(w_{i,n+1}) = \tilde{w}_{i,n} \mbox { for } i \neq n 
				\end{eqnarray*}
				(remember that the $n$th leaf of $\tt_n$ is not the $n$th leaf of $\tt_{n+1}$!)
			\item all other paths do not change
				\begin{eqnarray*} p_{n+1}(w_{i,j}) = \tilde{w_{i,j}} \mbox{ for } i,j \neq n,n+1 \end{eqnarray*}
		\end{itemize}
Therefore, $p_{n+1}$ maps paths in $\Snp$ to paths in $\Sn$. Every path in $\Sn$ is in the image, so $p_{n+1}$ restricts to a surjective homomorphism of semigroups from $\Snp$ to $\Sn$.\\


\begin{lem} Let $s = (s_1,\ldots, s_{2n-3}, s_{2n-2},s_{2n-1}) \in \Snp$. Then $s_{2n-2} = \beta + k, s_{2n-1} = \gamma + k$, where $k \in \mathbb{N}$, and $\beta + \gamma = s_{2n-3}$. \end{lem}
The statement of the lemma is equivalent to claiming that the sum of the values on the last two edges in $\tt_{n+1}$ must be greater or equal than the value on the preceding edge and differs from it by paths between last two leaves.
\begin{proof}
	Let $s = \sum_{1 \leq i < j \leq n+1} a_{i,j} w_{i,j}$ where $a_{i,j} \in \nn$. 
	
	Let $\beta = \sum_{i<n} a_{i,n}$ and $\gamma = \sum_{i<n} a_{i,n+1}$ be the number of paths different from the path between last two leaves and passing through the edges $2n-2$ and $2n-3$ respectively. 
	
	Obviously $s_{2n-2} = \beta + a_{n,n+1}$ and $s_{2n-1} = \gamma + a_{n,n+1}$. Each path crossing the edge $2n-3$ crosses either the edge $2n-1$ or the edge $2n-2$, so $\gamma + \beta = s_{2n-3}$. Take $k=a_{n,n+1}$.
\end{proof}

\begin{lem} 
	\label{sn}
It holds that
	\begin{eqnarray*}
		\Snp = \{ (s_1,s_2,\ldots, s_{2n-3}, \beta+k, \gamma+k) \mbox { } | \mbox { } (s_1,s_2,\ldots,s_{2n-3}) \in \Sn \mbox{, } k \in \nn \mbox{ and } \beta + \gamma = s_{2n-3}\}. \end{eqnarray*} \end{lem}
		We describe the fiber of the projection $p_{n+1}$.
\begin{proof}
	The last two lemmas imply that the elements of $\Snp$ must be of this form. We now prove that every element of this form belongs to the semigroup.\\

Let $s = (s_1,s_2,\ldots, s_{2n-3}, \beta +k, \gamma+k)$ for $k \in \nn$ and $\beta + \gamma = s_{2n-3}$. Let $(s_1,s_2,\ldots,s_{2n-3}) = \sum_{(i,j) \in P} \tilde{w}_{i,j}$ where $P$ is a multiset (the structure in which one element may occur many times) of ordered pairs of leaves of $\tt_n$. 

It holds that $s_{2n-3} = |\{(i,n) \in P \}|$ ($s_{2n-3}$ equals the number of elements of the form $(i,n)$ in $P$ counted with multiplicity). We divide this multiset $\{(i,n) \in P \}$ arbitrarily into two multisets $P,P'$ with $\beta$ and $\gamma$ elements respectively (recall that $\beta + \gamma$ is equal to the number of elements in the multiset counted with multiplicity).

We see that
	\begin{eqnarray*} s = \sum_{(i,j) \in P;\mbox{ } i,j <n} w_{i,j} + \sum_{(i,n) \in P'} w_{i,n} + \sum_{(i,n) \in P''} w_{i,n+1} + kw_{n,n+1} \end{eqnarray*}
Therefore $s \in \Snp$.

\end{proof}

Let $W_n$ be a Hilbert series for $X(\tt_n)$. Let us recall that $dim(\cc[X(\tt_n)]_{\lambda})$ is equal to the number of elements in $\Sn$ which belong to the gradation $\lambda$ (because $\cc[\Sn] \cong \cc[X(\tt_n)]$ and the algebra $\cc[\Sn]$, considered as a group, is free). 
\begin{theorem}
	Let $W_n = \sum_{i=0}^{\infty} w_i z_n^i$, where $w_i \in \zz[[z_1,\ldots,z_{n-1}]]$. Then for $n \geq 3$:
	\begin{eqnarray*}
		W_{n+1} = \left( \sum_{i=0}^{\infty} w_i \cdot (\sum_{l=0}^i z_n^{i-l} z_{n+1}^l)\right) \cdot \frac{1}{1 - z_n z_{n+1}}
	\end{eqnarray*} \label{rec_eq}
\end{theorem} \begin{proof}
The lemma $(\ref{grad})$ implies that the gradation of an element of a semigroup is the image of that element's projection to the space of leaves. For each element $s = (s_1,s_2,\ldots, s_{2n-3}) \in \Sn$, we must understand to which gradation the elements of $\Snp$ restricting to $s$ belong.

Let $(\alpha_1, \ldots, \alpha_n)$ be the gradation of the element $(s_1, \ldots, s_{2n-3}) \in \Sn$. Then the element \newline $(s_1,s_2,\ldots, s_{2n-3}, \beta +k, \gamma +k)$ where $\beta + \gamma = s_{2n-3}$ belongs to the gradation $(\alpha_1, \ldots, \alpha_{n-1}, \beta + k, \gamma +k)$. 

Therefore we get $W_{n+1}$ from $W_n$ by changing $z_n^i$ into the sum $\sum_{k \geq 0, \beta + \gamma = i} z_n^{\beta + k} z_{n+1}^{\gamma +k}$ (see lemma $\ref{sn}$). The theorem now follows from the equality: 
\begin{eqnarray*}
	&& \sum_{k \geq 0, \beta + \gamma = i} z_n^{\beta + k} z_{n+1}^{\gamma +k} = \left(\sum_{\beta + \gamma = i} z_n^{\beta} z_{n+1}^{\gamma}\right) \cdot \left(\sum_{k \geq 0} \left(z_n z_{n+1}\right)^k \right)  =
	\\ && \hspace{2cm} = \left(\sum_{l=0}^i z_n^{i-l} z_{n+1}^l \right) \cdot \frac{1}{1 - z_n z_{n+1}}
\end{eqnarray*}
\end{proof}
It is easy to see that the formula is also true for $n=2$.

\subsection{Other formulas for Hilbert series}
These formulas describe infinite objects (series), which is problematic because they don't allow "mechanical" calculations. The theorem regarding Hilbert-Poincare series implies that our series $W_n$ can be presented in the following form: 
\begin{eqnarray*} F_n(z_1,\ldots, z_n) / \prod_{1 \leq i < j \leq n} (1 - z_iz_j) \mbox{ where } F_n \in \zz[z_1,\ldots,z_n] \end{eqnarray*}
For performing calculations, the resursive formula for polynomials $F_n$ would be more valuable. Before we formulate it, let us define a few notions. Let:
\begin{eqnarray*}
F_n := \sum_{i=0}^d f_i z_n^i \mbox{, where } f_i \in \zz[z_1,\ldots,z_{n-1}] \mbox{ and } F_{n+1} = \sum_{i=0} \tilde{f_i} z_{n+1}^i
\end{eqnarray*}
Let $h_i$ be a sum of all monomials of degree $i$ and let $\sigma_i$ be the $i$-th elementary symmetric polynomial (both of $n-1$ variables $z_1,\ldots, z_{n-1}$). Let:
\begin{eqnarray*}
H_{s,l} := \sum_{r=0}^l h_{s-r} \cdot \sigma_r \cdot (-1)^r
\end{eqnarray*}

\begin{theorem} The following formulas hold: 
\begin{eqnarray*} 
&& \tilde{f}_t = \sum_{i=0}^d f_i a_{t-i,i} \mbox{, where: }\\
&& a_{k,l} = \sum_{\beta} z_n^{\beta} \sum_{\alpha=0}^{k+l} (-1)^{\alpha} \sigma_{\alpha} H_{k+\beta-\alpha, \beta}
\end{eqnarray*}
\end{theorem}

As the proof of the lemma is quite technical, we omit it in this version of the paper.

\section{Euler characteristic of sheaves on a plane with four blown-up points}

\subsection{The properties of a projective space with four blown-up points}
	This section consists of material which is well known, so we will omit proofs. \\

	Let $X'$ be a surface constructed from a surface $X$ by blowing it up at a point. We will denote a standard projection from $X'$ to $X$ by $\pi$. The Picard group of a variety will be denoted by $Pic$.

	Let $V$ be a projective plane $\cc\pp^2$ blown-up in four generic points $p_1$, $p_2$, $p_3$ and $p_4$ (such that no three are colinear). By $E_{0,i}$ (for $1 \leq i \leq 4$), we will denote the exceptional divisor corresponding to $p_i$ and let $E_{i,j}$ (for $1 \leq i < j \leq 4$) denote the lifting to $V$ of the line which passes through $p_i$ and $p_j$. The standard projection from $V$ onto a projective space will be denoted by $\pi_V$. For $1 \leq i < j \leq 4$, let us introduce notion $\overline{E_{i,j}} := E_{k,l}$ ($\{k,l\} = \{1,2,3,4\} \backslash \{i,j\}$) and $\overline{E_{0,i}} = E_{0,i}$ ($1 \leq i \leq 4$). 

The following propositions hold:
\begin{enumerate}
	\item Divisors $E_{0,i}$ and $E_{1,2}$ form a basis of $Pic(V)$.
	\item The canonical divisor $K_V$ for the surface $V$ is equal to:
		\begin{eqnarray*}
			-\frac{1}{2} \sum_{0 \leq i < j \leq 4} E_{i,j}
		\end{eqnarray*}
		Additionally, it satisfies: 
		\begin{eqnarray*}
			K_V . E_{i,j} = -1 \\
			{K_V}^2 = 5
		\end{eqnarray*}
\end{enumerate}

Let $D = \sum_{i=1}^4 a_{0,i}E_{0,i} + a_{1,2}E_{1,2}$.

The Riemann-Roch theorem for surfaces ($\cite{Ha}$, V.1.6) implies that:
	\begin{eqnarray*}
		\chi(\oo(D)) & = & \frac{1}{2}\cdot (D^2 - K\cdot D) + \chi(\oo(V)) =  \frac{1}{2}(D^2 - K\cdot D) + 1 = \\	
		           & = & \frac{1}{2}\cdot \bigg(\Big(-a_{0,1}^2 -a_{0,2}^2-a_{0,3}^2-a_{0,4}^2 - a_{1,2}^2 +2a_{1,2}(a_{0,1}+a_{0,2})\Big) + \\
				   &   & \hspace{1cm} 	a_{0,1}+a_{0,2}+a_{0,3}+a_{0,4}+a_{1,2}\bigg) + 1
	\end{eqnarray*}

\subsection{Grassmannian and a plane with four blown-up points}
In order to verify the formulas from the second chapter, we will calculate the Euler characteristic of invertible sheaves on $V$ in another way.\\ \\
It is a classical result (see the introduction to the paper $\cite{Muk}$) that:
\begin{eqnarray}
	\cc[\gg_5] = \bigoplus_{\lambda \in \zz^n} \cc[\gg_5]_{\lambda} \cong \bigoplus_{D \in Pic(V)} H^0\left(X,\oo(D)\right)  
	\label{kapranov}
\end{eqnarray}
This is an isomorphism of rings with gradation in $\zz^n$. The divisor $D = \sum a_{i,j}\overline{E_{i,j}} \in Pic(V)$ corresponds to the gradation $\lambda(D) = \sum a_{i,j}r_{i,j}$, where $r_{i,j} \in \zz^n$ has ones on the $i$th and $j$th coordinates and zeroes on the other coordinates. \\

\begin{theorem}
	Let $D = \sum_{0 \leq i < j \leq 4} a_{i,j} \overline{E_{i,j}}$ be a divisor such that nonzero cohomologies of $\oo(D)$ vanish. Then the coefficient of the monomial term $\prod (z_iz_j)^{a_{i,j}}$ in the Hilbert series $\gg_5$ (for the action of the torus described in the first chapter) is equal to $\chi(\oo(D))$. 
\end{theorem}
\begin{proof}
	The formula ($\ref{kapranov}$) implies that:
	\begin{eqnarray*}
		h^0(\oo(D)) = dim(\cc[{\gg}_5]_{\lambda(D)})
	\end{eqnarray*}
	As nonzero cohomologies of $\oo(D)$ vanish, it holds: $h^0(\oo(D)) = \chi(\oo(D))$.
\end{proof}

To verify formulas for Hilbert series, we will give some examples of divisors $D$ on the surface $V$, for which nonzero cohomologies vanish. 
\begin{lem} Nonzero cohomologies of $-2K + s_1E_{i,j} + s_2E_{k,l}$ vanish (for $s_1,s_2 \in \{1,-1\}$). \end{lem}
\begin{proof}
	At first we will show that $D = -3K + s_1 E_{i,j} + s_2 E_{k,l}$ is ample. It follows from the Nakai-Moishezon criterion ($\cite{Ha}$, V.1.10), which states that a divisor $D$ on a surface is ample if and only if $D^2> 0$ and $D.E > 0$ for every irreducible curve $E$ on this surface. Let us note that:
	\begin{eqnarray*}
		&& D^2 = 9K^2 - 3K(s_1E_{i,j} + s_2 E_{k,l}) + (s_1E_{i,j} + s_2 E_{k,l})^2 \geq 45 - 6 - 4 > 0 \\
		&& D.E = 3 + s_1 E_{i,j} . E + s_2 E_{k,l} . E \geq 1
	\end{eqnarray*}
	The Kodaira vanishing theorem ($\cite{Ha}$, III.7.15, $h^i(X, \oo(K) \otimes \oo(H)) = 0$ for $i>0$ and an ample divisor $H$) implies that if we add $K$ to the divisors mentioned above, we will receive a divisor whose nonzero cohomologies vanish. 
\end{proof}

We may calculate $\chi(\oo(D))$ in the following way:

The Riemann-Roch theorem shows that $\chi(\oo(D))$ for $D = \sum_{1\leq i < j \leq n} a_{i,j} E_{i,j}$ must be a quadratic form in variables $a_{i,j}$. Coefficients of this form may be calculated from the system of equations:
\begin{eqnarray*}
	\chi(\oo(D)) = \dim(\cc[{\gg}_5]_{\lambda(D)}) \mbox { for } D = -2K + s_1E_{i,j} + s_2E_{k,l}
\end{eqnarray*}
The Hilbert series obtained in the second chapter gives us dimensions of gradations. One can check that this system of equations has the required order: its solution gives us unique coefficients of the quadratic form we are looking for. Obviously, they are equal to the coefficients calculated directly through the Riemann-Roch theorem. The code in $Sage$ may be found in "Appendix".

\appendix

\section{Examples of Hilbert series for small dimensions and a program in Sage which calculates numerators of Hilbert series}
Hilbert series for small $n$: 
\begin{itemize}
	\item for $n=2$ 
		\begin{eqnarray*} \frac{1}{1-z_1z_2} \end{eqnarray*} 
	\item for $n=3$
		\begin{eqnarray*} \frac{1}{(1-z_1z_2)(1-z_2z_3)(1-z_3z_1)} \end{eqnarray*}
	\item for $n=4$
		\begin{eqnarray*}\frac{1-z_1z_2z_3z_4}{\prod_{1 \leq i < j \leq 4}(1 - z_iz_j)} \end{eqnarray*}
	\item for $n=5$
		\begin{eqnarray*}
		&&	(-z_0^2z_1^2z_2^2z_3^2z_4^2 + z_0^2z_1z_2z_3z_4 + z_0z_1^2z_2z_3z_4 + z_0z_1z_2^2z_3z_4 + z_0z_1z_2z_3^2z_4 + z_0z_1z_2z_3z_4^2 - \\
		&& z_0z_1z_2z_3 - z_0z_1z_2z_4 - z_0z_1z_3z_4 - z_0z_2z_3z_4 - z_1z_2z_3z_4 + 1) / \prod_{1 \leq i < j \leq 5} (1 - z_iz_j) \end{eqnarray*}

\end{itemize}
\begin{verbatim}
Homogenous = SFAHomogeneous(QQ)
Elementary = SFAElementary(QQ)

def e(n,deg):
    return Polynomial(Elementary([deg]).expand(n)) 
		if deg > 0 else (1 if deg == 0 else 0)
def h(n, deg):
    return Polynomial(Homogenous([deg]).expand(n)) 
		if deg > 0 else (1 if deg == 0 else 0)
def H(n,s,l):
    return sum(map(lambda r: h(n,s-r)*e(n,r)*(-1)^r, 
                             range(0,l+1)))
                             
the def a(n,k,l):
    def a_help(n,k,l,b):
        return sum(map(lambda a: ((-1)^a)*e(n,a)*H(n,k+b-a,b), 
                                 range(0, k+l+1)))
    return sum(map(lambda b: (z[n]^b)*a_help(n,k,l,b), range(0,n)))

# main_hilbert calculates coefficients aij
def main_hilbert(size, n):
    if n == 2 or n==1:
        return Matrix(Polynomial, size, 1, lambda i,j: 1 if i==0 else 0)
    M = Matrix(Polynomial, size, size, lambda i,j: a(n-1, i-j, j))
    return M*main_hilbert(size, n-1) 
    
# find_hilbert_series finds the numerators of the Hilbert series           
def find_hilbert_series(n):
    v = [Matrix(Polynomial, n+1, 1, lambda i,j: z[n]^i) for n in range(0,n+1)]
    return (v[n].transpose()*main_hilbert(n+1, n))[0][0]

%Example
num_var = 5
Polynomial = PolynomialRing(QQ, num_var, 'z')
z = PolynomialRing(QQ, num_var, 'z').gens()
res = find_hilbert_series(num_var-1)

\end{verbatim}

\section{A program in Sage which calculates the Euler characteristic of invertible sheaves on a spaces with four blown-up points}

\begin{verbatim}
num_var = 5; precision = 25
num_pairs = Integer(num_var*(num_var-1)/2)
z = PolynomialRing(QQ, num_var, 'z').gens()

# function base_change takes coefficients of a divisor \sum a_{i,j} E_{i,j} 
# and returns the coefficients in the basis E_{0,1},E_{0,2},E_{0,3}, E_{0,4},E_{1,2}. 
# Calculations are based on the proposition 3.1.3
def base_change(a01,a02,a03,a04,a12,a13,a14,a23,a24,a34):
    return [a01 + a23 + a24 + a34, a02 + a13 + a14 + a34, 
            a03-a13-a23-a34, 
            a04-a14-a24-a34,  
            a12+a13+a14+a23+a24+a34]

# kapranov_iso function takes coefficients of the divisor in basis
# E_{0,1},E_{0,2},E_{0,3}, E_{0,4},E_{1,2} and returns the gradation
# of the corresponding monomial of Grassmannian
def kapranov_iso(a01,a02,a03,a04,a12):
    return [a01+a02+a03+a04, a01, a02, a03+a12, a04+a12]

# calculation of the denominator of the Hilbert Series
denom = reduce(lambda x,y: x*y, [1 - z[i]*z[j] 
                                 for i in range(num_var) 
                                 for j in range(i+1,num_var)])
 
# fixing precision 
P = PowerSeriesRing(QQ, 5, "z", default_prec = precision)
# calculation of Hilbert series 
H = (P)(-z[0]^2*z[1]^2*z[2]^2*z[3]^2*z[4]^2 + z[0]^2*z[1]*z[2]*z[3]*z[4] + 
         z[0]*z[1]^2*z[2]*z[3]*z[4] + z[0]*z[1]*z[2]^2*z[3]*z[4] + 
         z[0]*z[1]*z[2]*z[3]^2*z[4] + z[0]*z[1]*z[2]*z[3]*z[4]^2 - 
         z[0]*z[1]*z[2]*z[3] - z[0]*z[1]*z[2]*z[4] - z[0]*z[1]*z[3]*z[4] - 
         z[0]*z[2]*z[3]*z[4] - z[1]*z[2]*z[3]*z[4] + 1)/denom

# calculation of the dictionary of coefficients of H
Hdict = H.dict() 

# We are looking for a quadratic form in variables a01,a02,a03,a04,a12 
# monomial_values returns values of subsequent monomials 
# of degree less or equal to two in variables a01,a02,a03,a04,a12 
def monomial_values(a):
    b = [1] + a
	    return [b[i]*b[j] for i in range(num_var+1)
                      for j in range(i,num_var+1)]

# base_vector returns a vector of size num_pairs, which has a one
# at the i-th position and zeroes at all other positions
def base_vector(i):
    v = vector(QQ,num_pairs)
    v[i] = 1
    return v
# one_vector returns a vector of size num_pairs consisting only of ones
def one_vector():
    return vector(QQ, [1 for i in range(num_pairs)])
				 
# list_of_divs returns ununique coefficients of chosen 
# divisors (-2K +- E_{i,j} +- E_{k,l})
def list_of_divs():
    return [one_vector() + s1*base_vector(k) + s2*base_vector(l) 
            for k in range(num_pairs) 
            for l in range(k,num_pairs) 
            for s1 in [1,-1] 
            for s2 in [1,-1]]

# div_into_eq takes the divisor's coefficients and returns a vector 
# with values of the monomials 
def div_into_eq(a):
    return monomial_values(base_change(*a)) 

# eqmatrix is a matrix of the system of equations
eqmatrix = Matrix(map(lambda a : div_into_eq(a), list_of_divs()))  

# solution is a vector of values of the quadratic form
solution = vector(map(lambda a: Hdict[
    sage.rings.polynomial.polydict.ETuple(
	        kapranov_iso(*base_change(*a)))], list_of_divs())) 

eqmatrix.solve_right(solution)
\end{verbatim}

\bibliography{eng_praca}

\end{document}